\documentclass[a4paper,12pt]{article}

\usepackage{amsmath,amsfonts,amssymb}
\usepackage{amsthm}
\usepackage{color,multicol}
\usepackage{graphicx}

\newtheorem{lemma}{Lemma}[section]
\newtheorem{theorem}[lemma]{Theorem}

\newtheorem{definition}[lemma]{Definition}

\newtheorem{remark}[lemma]{Remark}

\begin{document}

\title{Eigenvalue bounds for a class of Schr\"odinger operators in a strip  }

\author{Martin Karuhanga\\Department of Mathematics,\\ Mbarara University of Science and Technology,\\P.O BOX 1410, Mbarara, Uganda\\E-mail: mkaruhanga@must.ac.ug}

\date{}

\maketitle

\begin{abstract}
This paper is concerned with the estimation of the number of negative eigenvalues (bound states) of Schr\"odinger operators in a strip subject to Neumann boundary conditions. The estimates involve weighted $L^1$ norms and $L\ln L$ norms of the potential.  Estimates involving the norms of the potential supported by a curve embedded in a strip are also presented.
\end{abstract}

\section{Introduction}
The celebrated Cwikel-Lieb-Rozenblum (CLR) inequality \cite{Roz} gives an upper bound for the number of negative eigenvalues of the Schr\"odinger operator $-\Delta - V,\,V\ge 0$ on $L^2(\mathbb{R}^d),\,d > 2$. It is known that the CLR inequality does not hold for $d = 2$ and one of the reasons is that the Sobolev space $H^1(\mathbb{R}^2)$ is not continuously embedded in $L^{\infty}(\mathbb{R}^2)$. However, $H^1(\mathbb{R}^2)\hookrightarrow L^q(\mathbb{R}^2)$ for all $q\in [2, +\infty)$ and there are estimates involving $\int_{\mathbb{R}^2}|V(x)|^r\,dx , \, \forall r > 1$ (see, e.g., \cite{BL, Grig, LM}).
More precisely, $H^1(\mathbb{R}^2)$ is embedded in a space of exponentially integrable functions which lies between $L^1(\mathbb{R}^2)$ and $L^p(\mathbb{R}^2),\,p>1$ (see, e.g., \cite{Ad}). This gives rise to estimates involving a norm of $V$ weaker than $\|V\|_{L^r},\, r>1$, namely, the Orlciz $L\ln L$ norm (see, e.g., \cite{Eugene, Sol}). The strongest known estimates have been obtained in \cite{Eugene}. For more information regarding upper estimates for the number of negative eigenvalues of two-dimensional Schr\"odinger operators refer to \cite{BL, FL, Grig, Kar, LM, Eugene, Sol} and the references therein.
This paper provides estimates for the number of negative eigenvalues of the Schr\"odinger operator $-\Delta - V$ on $L^2(S)$ whose domain is characterized by the Neumann boundary conditions, where $S$ is an infinite straight strip . We use the results of Shargorodsky \cite{Eugene} to obtain improved versions of the estimates by Grigor'yan and Nadirashvili \cite{Grig}. This improvement is achieved by replacing $\|V\|_{L^p},\, p > 1$ in the estimates of \cite[Section 7]{Grig} by the $L \ln L$ norms of $V$. In addition, these estimates are extended to the case of strongly singular potentials (see Section \ref{singular}). The precise description of the operator here studied is as follows:\\
Let $S := \{(x_1, x_2)\in\mathbb{R}^2\;:\;x_1\in\mathbb{R},\; 0 < x_2 < a\},\;a > 0$ and $V : \mathbb{R}^2 \longrightarrow \mathbb{R}$  be a function integrable on bounded subsets of $S$. Consider the following self-adjoint operator on $L^2(S)$
\begin{equation}\label{self-adj}
H_V = -\Delta - V, \;\;\;V\ge 0,
\end{equation}
with homogeneous Neumann boundary conditions both at $x_2 = 0$ and $x_2 = a$. The main objective of this paper is to obtain estimates for the number of negative eigenvalues of $\eqref{self-adj}$ in terms of the norms of $V$. \\\\
The strategy used here is as follows: The problem is split into two problems. The first one is defined by the restriction of the quadratic form associated with the operator \eqref{self-adj} to the subspace of functions of the form $w(x_1)u_1(x_2)$, where $u_1$ is the first eigenfunction of the one-dimensional differential operator on $L^2((0, a))$ and hence, is  reduced to a well studied one-dimensional Schr\"odinger operator with the potential equal to a weighted mean value $\widetilde{V}$ of $V$ over $(0, a)$. The second problem is defined by a class of functions orthogonal to constant functions in the $L^2((0, a))$ inner product.

\section{Preliminaries}
Let $(\Omega, \Sigma, \mu)$ be a measure space and let $\Psi : [0, +\infty) \rightarrow [0, +\infty)$ be a non-decreasing function. The Orlicz class $K_{\Psi}(\Omega)$ is the set of all of measurable functions $f : \Omega \rightarrow \mathbb{C}\;( \textrm{or}\;\mathbb{R})$ such that
\begin{equation}\label{orliczeqn}
\int_{\Omega}\Psi(|f(x)|)d\mu(x) < \infty\,.
\end{equation} If $\Psi(t) = t^p,\; 1\le p < \infty$, this is just the $L^p(\Omega)$ space. The \textit{Orlicz space} $ L_{\Psi}(\Omega)$ is the linear span of the Orlicz class $K_{\Psi}(\Omega)$, that is, the smallest vector space containing $K_{\Psi}(\Omega)$.
\begin{definition}
{\rm A continuous non-decreasing convex function $\Psi : [0, +\infty) \rightarrow [0, +\infty)$ is called an $N$-function if
$$
\underset{t \rightarrow 0+}\lim\frac{\Psi (t)}{t} = 0 \;\;\; \textrm{and }\;\;\;\underset{t \rightarrow \infty}\lim\frac{\Psi (t)}{t} = \infty.
$$ The function $\Phi : [0, +\infty) \rightarrow [0, +\infty)$ defined by
$$
\Phi(t) := \underset{s\geq 0}\sup\left(st - \Psi(s)\right)
$$ is called complementary to $\Psi$.
}
\end{definition}
\begin{definition}
{\rm An $N$-function $\Psi$ is said to satisfy a global $\Delta_2$-condition if there exists a positive constant $k$ such that for every $t \geq 0$,
\begin{equation}\label{global}
\Psi(2t)\le k\Psi(t).
\end{equation}
Similarly $\Psi$ is said to satisfy a $\Delta_2$-condition near infinity if there exists $t_0 > 0$ such that \eqref{global} holds for all $t \geq t_0$.}
\end{definition}

Let $\Phi$ and $\Psi$ be mutually complementary $N$-functions, and let $L_\Phi(\Omega)$,
$L_\Psi(\Omega)$ be the corresponding Orlicz spaces.  We will use the following
norms on $L_\Psi(\Omega)$
\begin{equation}\label{Orlicz}
\|f\|_{\Psi} = \|f\|_{\Psi, \Omega} = \sup\left\{\left|\int_\Omega f g d\mu\right| : \
\int_\Omega \Phi(|g|) d\mu \le 1\right\}
\end{equation}
and
\begin{equation}\label{Luxemburg}
\|f\|_{(\Psi)} = \|f\|_{(\Psi, \Omega)} = \inf\left\{\kappa > 0 : \
\int_\Omega \Psi\left(\frac{|f|}{\kappa}\right) d\mu \le 1\right\} .
\end{equation}
These two norms are equivalent
\begin{equation}\label{Luxemburgequiv}
\|f\|_{(\Psi)} \le \|f\|_{\Psi} \le 2 \|f\|_{(\Psi)}\, , \ \ \ \forall f \in L_\Psi(\Omega),
\end{equation}(see \cite{Ad}).\\
Note that
\begin{equation}\label{LuxNormImpl}
\int_\Omega \Psi\left(\frac{|f|}{\kappa_0}\right) d\mu \le C_0, \ \ C_0 \ge 1  \ \ \Longrightarrow \ \
\|f\|_{(\Psi)} \le C_0 \kappa_0 .
\end{equation}
It follows from \eqref{LuxNormImpl} with $\kappa_0 = 1$ that
\begin{equation}\label{LuxNormPre}
\|f\|_{(\Psi)} \le \max\left\{1, \int_{\Omega} \Psi(|f|) d\mu\right\} .
\end{equation}

We will need the following
equivalent norm on $L_\Psi(\Omega)$ with $\mu(\Omega) < \infty$ which was introduced in
\cite{Sol}:
\begin{equation}\label{OrlAverage}
\|f\|^{\rm (av)}_{\Psi} = \|f\|^{\rm (av)}_{\Psi, \Omega} = \sup\left\{\left|\int_\Omega f g d\mu\right| : \
\int_\Omega \Phi(|g|) d\mu \le \mu(\Omega)\right\} .
\end{equation}
We will use the following pair of pairwise complementary $N$-functions
\begin{equation}\label{thepair}
\mathcal{A}(s) = e^{|s|} - 1 - |s| , \ \ \ \mathcal{B}(s) = (1 + |s|) \ln(1 + |s|) - |s| , \ \ \ s \in \mathbb{R} .
\end{equation}
Let $I_1, I_2 \subseteq \mathbb{R}$ be nonempty open intervals.
We denote by $L_1\left(I_1, L_{\mathcal{B}}(I_2)\right)$ the space of measurable functions
$f : I_1\times I_2 \to \mathbb{C}$ such that
\begin{equation}\label{L1LlogLnorm}
\|f\|_{L_1\left(I_1, L_{\mathcal{B}}(I_2)\right)} := \int_{I_1}
\|f(x, \cdot)\|_{\mathcal{B}, I_2}\, dx < +\infty\,.
\end{equation}
Let us recall that a sequence $\{a_n\}$ belongs to the ``weak $l_1$-space'' (Lorentz space) $l_{1,w}$ if the following quasinorm
\begin{equation}\label{quasi}
\|\{a_n\}\|_{1,w} = \underset{s > 0}\sup\left(s \;\textrm{card}\{n\;:\;|a_n| > s\}\right)
\end{equation} is finite. It is a quasinorm in the sense that it satisfies the weak version of the triangle inequality:
$$
\|\{a_n\} + \{b_n\}\|_{1,w} \le 2\left(\|\{a_n\}\|_{1,w} + \|\{b_n\}\|_{1,w}\right).
$$

The quasinorm \eqref{quasi} induces a topology on $l_{1,w}$ in which this space is non-separable. The closure of the set of elements $a_n$ with only finite number of non-zero terms is a separable subspace in $l_{1,w}$.
It is well known that $l_1 \subset l_{1,w}$ and
$$
\|\{a_n\}\|_{1,w} \le \|\{a_n\}\|_{l_1}
$$ (see, e.g., \cite{BirSol} for more details).

\section{Estimating the number of negative eigenvalues in a strip}
Define \eqref{self-adj} via its quadratic form
\begin{eqnarray*}
q_{V,S}[u]&:=& \int_{S}|\nabla u(x)|^2dx - \int_{S}V(x)|u(x)|^2\,dx,\\
\textrm{Dom}(q_{V, S}) &=& W^1_2(S)\cap L^2(S, Vdx),
\end{eqnarray*}where $W^1_2(S)$ denotes the standard Sobolev space $H^1(S)$.
Let $N_-(q_{V, S})$ denote the number of negative eigenvalues of $\eqref{self-adj}$ repeated according to their multiplicities.
Then $N_-(q_{V, S})$ is given by
\begin{equation}\label{negative}
N_-(q_{V, S})= \underset{L}\sup\{\textrm{dim} \,L \,:\,q_{V,S}[u] < 0,\,\forall u\in L\setminus \{0\}\},
\end{equation}
where $L$ denotes a linear subspace of $\textrm{Dom}(q_{V, S})$ (see, e.g., \cite[Theorem 10.2.3]{BirSol}).\\Let
\begin{eqnarray*}
I_n := [2^{n - 1}, 2^n], \ n > 0 , \ \ \ I_0 := [-1, 1] , \ \ \
I_n := [-2^{|n|}, -2^{|n| - 1}], \ n < 0
\end{eqnarray*} and
\begin{equation}\label{calAn*}
\mathcal{G}_n := \int_{I_n}\int_0^a |x_1| V(x)\, dx , \ n \not= 0 , \ \ \ \mathcal{G}_0 := \int_{I_0}\int_0^a V(x)\, dx .
\end{equation}
Furthermore, let
$$J_n:= (n ,n+1), \,n\in\mathbb{Z}, \,I := (0, a)\mbox{ and } \mathcal{D}_n := \parallel V\parallel_{L_1(J_n, L_\mathcal{B}(I))}.$$ Then we have the following result.
\begin{theorem}\label{thm1}
There exist constants $c, C > 0$ such that
\begin{equation}\label{Gest2}
N_-(q_{V, S})\leq 1 + C\left(\sum_{\{n\in\mathbb{Z}, {\mathcal{G}_n > c\}}}\sqrt{\mathcal{G}_n} + \sum_{\{n\in\mathbb{Z}, \mathcal{D}_n > c\}}\mathcal{D}_n\right), \;\;\;\;\forall V\geq 0.
\end{equation}
\end{theorem}

Let $$L_1 := \{u\in L^2(S) \; : \; u(x) = u(x_1)\}$$ and  $P : L^2(S)\longrightarrow L_1$ be a projection defined by  $$Pv(x) := \frac{1}{a}\int^a_0 v(x)dx_2 = Pv(x_1).$$
Indeed, $P$ is a projection since $P^2 = P$. Let $L_2 := (I - P)L^2(S)$, then one can show that $L^2(S) = L_1 \oplus L_2$. Here and below $\oplus$ denotes a direct orthogonal sum.\\ Indeed for all $v\in L_2$ we have,
\begin{eqnarray*}
\int^a_0 v(x)dx_2 &=& \int^a_0 (I - P)v(x)dx_2\\ &=& \int^a_0 v(x)dx_2 - \int^a_0 Pv(x)dx_2\\ &=&\int^a_0 v(x)dx_2 - \int^a_0 Pv(x_1)dx_2\\&=& \int^a_0 v(x)dx_2 - \int^a_0 v(x)dx_2\\ &=& 0.
\end{eqnarray*} Now pick $w\in L_1$ and $v\in L_2$, then,
\begin{eqnarray*}
\langle v, w\rangle_{L^2(S)} = \int_S v(x)\overline{w(x_1)}\,dx &=&\int_{\mathbb{R}}\left(\int^a_0 v(x)\,dx_2\right)\overline{w(x_1)}dx_1 \\&=&0.
\end{eqnarray*}
Similarly, let
\begin{equation}\label{H}
 \mathcal{H}_1 := \{u\in W^1_2(S) \; : \; u(x) = u(x_1)\} \; \textrm{and}\; \mathcal{H}_2 := (I- P)W^1_2(S),
\end{equation}
 then $$W^1_2(S) = \mathcal{H}_1 \oplus \mathcal{H}_2.$$ Indeed, for all $v\in\mathcal{H}_1$ and all $w\in\mathcal{H}_2$ we have $$\langle v, w\rangle_{W^1_2(S)} =  \int_S \left(v(x_1)\overline{w(x)} +  v_{x_1}(x_1)\overline{w_{x_1}(x)} +  v_{x_2}(x_1)\overline{w_{x_2}(x)}\right)dx = 0.$$ This is so because $v, v_{x_1}\in L_1, \overline{w}, \overline{w_{x_1}}\in L_2$ and $v_{x_2} = 0$. To see this note that $v(x_1)$ and $v_{x_1}(x_1)$ do not depend on $x_2$ implying that $v_{x_1}\in L_1$. Also, $w\in L_2\Leftrightarrow \int^a_0w(x)dx_2 = 0$. So, $\frac{d}{dx_1}\int^a_0w(x)\,dx_2 = 0 \Rightarrow \int^a_0w_{x_1}(x)\,dx_2 = 0 \Rightarrow w_x\in L_2$. Hence $\int_Sv_x(x)\overline{w_x(x, y)}dx_1dx_2 = 0$.\\\\ Now  for all $u\in W^1_2(S)$,  $u = v + w, \;\; v\in\mathcal{H}_1,\;\;w\in\mathcal{H}_2$ one has
\begin{eqnarray*}
\int_S \mid\nabla u(x)\mid^2\,dx &=& a\int_{\mathbb{R}}\mid v'(x_1)\mid^2\,dx_1 + \int_S\mid\nabla w(x)\mid^2\,dx\\ &+& \underbrace{\int_S\nabla v(x_1).\overline{\nabla w(x)}\,dx + \int_S\nabla w(x).\overline{\nabla v(x_1)}\,dx}_{= 0}\\ &=&a\int_{\mathbb{R}}\mid v'(x_1)\mid^2\,dx_1 + \int_S\mid\nabla w(x)\mid^2\,dx
\end{eqnarray*}and
\begin{eqnarray*}
\int_S V(x)\mid u(x)\mid^2\,dx &=& \int_SV(x)\mid v(x_1)\mid^2\,dx + \int_SV(x)\mid w(x)\mid^2\,dx \\&+& \int_S 2V(x)\textrm{Re}(v.w)dx\\&\le& 2\int_{\mathbb{R}}\widetilde{V}(x_1)\mid v(x_1)\mid^2\,dx_1 + 2\int_SV(x)\mid w(x)\mid^2\,dx\,
\end{eqnarray*} where
$$\widetilde{V}(x_1) = \frac{1}{a}\int^a_0V(x)dx_2\,.$$
So
\begin{eqnarray*}
&&\int_S \mid\nabla u(x)\mid^2 \,dx - \int_S V(x)\mid u(x)\mid^2 \,dx\\ &&\geq \int_\mathbb{R} \mid v'(x_1)\mid^2 \,dx_1 - 2\int_\mathbb{R} \widetilde{V}(x_1)\mid v(x_1)\mid^2\, dx_1 \\&&+ \int_S \mid\nabla w(x)\mid^2\, dx - 2\int_S V(x)\mid w(x)\mid^2\, dx\,.
\end{eqnarray*}  Hence
\begin{equation}\label{estimate}
N_-(q_{V, S}) \leq N_-(q_{1, 2\widetilde{V},\mathbb{R}}) + N_-(q_{2, 2V, S}),
\end{equation}where $q_{1, 2\widetilde{V},\mathbb{R}}$ and $q_{2, 2V, S}$ denote the restrictions of the form $q_{2V, S}$ to the spaces $\mathcal{H}_1$ and $\mathcal{H}_2$ respectively.\\Let
\begin{equation}\label{calAn}
\mathcal{G}_n := \int_{I_n} |x_1| \widetilde{V}(x_1)\, dx_1 , \ n \not= 0 , \ \ \ \mathcal{G}_0 := \int_{I_0} \widetilde{V}(x_1)\, dx_1 .
\end{equation} Then similarly to the estimate before (39) in \cite{Eugene} one has
\begin{equation}\label{Est1*}
N_-(q_{1, 2\widetilde{V},\mathbb{R}})\leq 1 + 7.61\sum_{\{n\in\mathbb{Z},\; \mathcal{G}_n > 0.046\}}\sqrt{\mathcal{G}_n}\,.
\end{equation}
In terms of the original potential $V$
\begin{eqnarray*}
\mathcal{G}_n = \int_{I_n} |x_1| \widetilde{V}(x_1)\, dx_1 &=& \int_{I_n}|x_1|\left(\int_0^a V(x)\,dx_2\right)\,dx_1\\ &=& \int_{I_n\times (0, a)}|x_1|V(x)\,dx
\end{eqnarray*} and
\begin{eqnarray*}
\mathcal{G}_0 = \int_{I_0} \widetilde{V}(x_1)\, dx_1 &=& \int_{I_0}\left(\int_0^a V(x)\,dx_2\right)\,dx_1\\ &=& \int_{I_0\times (0, a)}V(x)\,dx.
\end{eqnarray*}
It now remains to find an estimate for $N_-(q_{2, 2V, S})$ in \eqref{estimate}.\\

Let $S_n := J_n\times I ,\; n\in\mathbb{Z}$, where $J_n := (n, n + 1)$ and $I := (0, a)$.  Then the variational principle (see, e.g., \cite[Ch.6, $\S$ 2.1, Theorem 4]{Cour}) implies
\begin{equation}\label{estimate2}
N_-(q_{2, 2V, S})\leq \sum_{n\in\mathbb{Z}}N_-(q_{2, 2V, S_n}),
\end{equation}
where
\begin{eqnarray*}
&& q_{2, 2V, S_n}[w] := \int_{S_n} |\nabla w(x)|^2\, dx  -
2\int_{S_n} V(x) |w(x)|^2\, dx , \\
&& \mbox{Dom}\, (q_{2, 2V, S_n}) =
\left\{w\in W^1_2(S_n)\cap L^2\left(S_n, V(x)dx\right) : \int_{S_n}w(x)\, dx = 0\right\}.
\end{eqnarray*}

\begin{lemma}\label{lemma6}{\rm (cf. \cite[Lemma 7.8]{Eugene})}
There exists $C_{1} > 0$ such that
\begin{equation}\label{CLRl4est}
N_- (q_{2, 2V, S_n}) \le C_{1}
\|V\|_{L_1\left(J_{n}, L_{\mathcal{B}}(I)\right)} , \ \ \ \forall V \ge 0
\end{equation}(see \eqref{L1LlogLnorm}).
\end{lemma}
\begin{proof}
Let $I= \mathbb{I}$, the unit interval. Then it follows from Lemma \cite[Lemma 7.7]{Eugene} that there is a constant $d_1> 0$ such that
\begin{equation}\label{d}
N_- (q_{2, 2V, S_n}) \le d_{1}
\|V\|_{L_1\left(J_{n}, L_{\mathcal{B}}(\mathbb{I})\right)} + 1 , \ \ \ \forall V \ge 0.
\end{equation}
Similarly to $(62)$ in \cite{Eugene} and using the Poincar\'{e} inequality (see, e.g., \cite[1.1.11]{Maz}), there is a constant $d_2 >0$ such that
$$
2\int_{S_n} V(x) |w(x)|^2\, dx \le d_2 \|V\|_{L_1\left(J_{n}, L_{\mathcal{B}}(\mathbb{I})\right)}\int_{S_n} |\nabla w(x)|^2\, dx
$$
for all $w\in W^1_2(S_n)\cap L^2\left(S_n, V(x)dx\right)$ such that $\int_{S_n}w(x)\, dx = 0$. If $\|V\|_{L_1\left(J_{n}, L_{\mathcal{B}}(\mathbb{I})\right)} \le \frac{1}{d_2}$, then $N_- (q_{2, 2V, S_n}) = 0$. Otherwise \eqref{d} implies
$$
N_- (q_{2, 2V, S_n}) \le C_{1}
\|V\|_{L_1\left(J_{n}, L_{\mathcal{B}}(\mathbb{I})\right)} , \ \ \ \forall V \ge 0,
$$
where $C_1 := d_1 + d_2$. Hence \eqref{CLRl4est} follows by the scaling $x_2 \longrightarrow a x_2$.
\end{proof}

\textbf{Proof of Theorem \ref{thm1}}
 \begin{proof}
 If $\mathcal{D}_n < \frac{1}{C_{1}}$ , then  $N_-(q_{2, 2V, S_n}) = 0$ and one can drop this term from the sum \eqref{estimate2}. Hence for any $c < \frac{1}{C_{1}}$ , \eqref{estimate2}  and Lemma \ref{lemma6} imply that
$$
N_-(q_{2, 2V, S})\leq C_{1}\sum_{\{n\in\mathbb{Z}, \mathcal{D}_n > c\}}\mathcal{D}_n  \;\;\;\;\; \; \forall V \geq 0.
$$ This together with \eqref{estimate} and \eqref{Est1*} imply  \eqref{Gest2}.
\end{proof}

One can easily show that \eqref{Gest2} is an improvement of the estimates obtained  by A. Grigor'yan and N. Nadirashvili  \cite[Theorem 7.9]{Grig} with a different $c$ and that \eqref{Gest2} is strictly sharper. Indeed, let
$B_{n}  := \parallel V\parallel_{\mathcal{B}, S_n}$. Then it follows from the embedding $L^p(S_n)\hookrightarrow L_{\mathcal{B}}(S_n)$ that there is a constant $C(p),\; p > 1$ such that
$$
B_{n}  = \parallel V\parallel_{\mathcal{B}, S_n} \le C(p)\left(\int_{S_n}V(x)^p\,dx\right)^{\frac{1}{p}} = C(p)b_n(V),
$$ where $b_n(V) := \left(\int_{S_n}V(x)^p\,dx\right)^{\frac{1}{p}}$
(see \cite[Remark 6.3]{Eugene}).
Now it follows from the known theory of embeddings of mixed-norm Orlicz spaces (see, e.g., \cite{bo, fi}) that $$\mathcal{D}_n \le C(p)b_n(V).$$
Hence
\begin{equation}\label{equivnad}
N_- (q_{2, 2V, S_n}) \le C_{2} b_n(V) , \ \ \ \forall V \ge 0\,,
\end{equation}where $C_{2} := C_{1}C(p)$.
The scaling $V \longmapsto t V,\; t > 0$, allows one to extend the above inequality to an arbitrary $V \geq 0$. Thus for any $c < \frac{1}{C_{2}}$, \eqref{Gest2} implies \cite[Theorem 7.9]{Grig}.\\\\
Next we will discuss different  forms of \eqref{Gest2}.
\begin{remark}
{\rm Note that
\begin{equation}\label{weak}
\sum_{\{n\in\mathbb{Z}, \mathcal{G}_n > c\}}\sqrt{\mathcal{G}_n} \leq \frac{2}{\sqrt{c}}\parallel(\mathcal{G}_n)_{n\in\mathbb{Z}}\parallel_{1,w}
\end{equation}
(see (49) in \cite{Eugene}).
Estimate \eqref{Gest2} implies the following estimate
\begin{equation}\label{Est2}
N_-(q_{V, S})\leq 1 + C_{3}\left( \parallel(\mathcal{G}_n)_{n\in\mathbb{Z}}\parallel_{1,w} + \|V\|_{L_1\left(\mathbb{R}\;, L_{\mathcal{B}}(I)\right)}\right), \;\;\;\ \forall V \geq 0.
\end{equation}
This follows from \eqref{weak}
 and
 $$
\sum_{\{n\in\mathbb{Z}, \mathcal{D}_n > c\}}\mathcal{D}_n \leq \sum_{n\in\mathbb{Z} }\mathcal{D}_n = \int_\mathbb{R}\parallel V(x_1, .)\parallel_{\mathcal{B}, I}dx_1 = \|V\|_{L_1\left(\mathbb{R}, L_{\mathcal{B}}(I)\right)}.
$$
Equation \eqref{Est2} in turn implies the following
\begin{equation}\label{Est3}
N_-(q_{V, S})\leq 1 + C_{4}\left( \parallel(\mathcal{G}_n)_{n\in\mathbb{Z}}\parallel_{1,w} + \int_{\mathbb{R}}\left(\int_{I}\mid V(x)\mid^p dx_2\right)^{\frac{1}{p}}dx_1\right), \;\;\;\ \forall V \geq 0,
\end{equation}
which is equivalent to

\begin{equation}\label{Est4}
N_-(q_{V, S})\leq 1 + C_{5}\left( \parallel(\mathcal{G}_n)_{n\in\mathbb{Z}}\parallel_{1,w} + \int_{\mathbb{R}}\left(\int_{I}\mid V_\ast(x )\mid^p dx_2\right)^{\frac{1}{p}}dx_1\right), \;\;\;\ \forall V \geq 0,
\end{equation}
where $V_{\ast}(x) = V(x) - \widetilde{V}(x_1), \;\; \widetilde{V}(x_1) = \frac{1}{a}\int^a_0 V(x)dx_2$. \\Indeed,
\begin{eqnarray*}
&&\left|\int_{\mathbb{R}}\left(\int_{I}\mid V(x)\mid^p dx_2\right)^{\frac{1}{p}}dx_1 - \int_{\mathbb{R}}\left(\int_{I}\mid V_\ast(x)\mid^p dx_2\right)^{\frac{1}{p}}dx_1\right| \\&\leq& \int_{\mathbb{R}}\left(\int_{I}\mid \widetilde{V}(x_1)\mid^p dx_2\right)^{\frac{1}{p}}dx_1 = a^{\frac{1}{p}}\int_{\mathbb{R}}\widetilde{V}(x_1)dx_1\\ &=& a^{\frac{1}{p}}\sum_{n\in\mathbb{Z}}\int_{I_n\times (0, a)}V(x)dx \le  a^{\frac{1}{p}}\sum_{n\in\mathbb{Z}}2^{-|n| + 1}\mathcal{G}_n\\&\leq&\textrm{const}\; \underset{{n\in\mathbb{Z}}}\sup\, \mathcal{G}_n \leq \textrm{const}\parallel(\mathcal{G}_n)_{n\in\mathbb{Z}}\parallel_{1,w}.
\end{eqnarray*}
Thus \eqref{Est3} and \eqref{Est4} are equivalent.\\\\Similarly,
\begin{eqnarray*}
&&\left| \|V\|_{L_1\left(\mathbb{R}, L_{\mathcal{B}}(I)\right)} - \|V_{\ast}\|_{L_1\left(\mathbb{R}, L_{\mathcal{B}}(I)\right)}\right|\\ &&= \left| \int_{\mathbb{R}}\parallel V(x_1, .)\parallel_{\mathcal{B},I}dx_1 - \int_{\mathbb{R}}\parallel V_{\ast}(x_1, .)\parallel_{\mathcal{B},I}dx_1\right|\\&&\le \int_{\mathbb{R}}\|\widetilde{V}(x_1)\|_{\mathcal{B}, I}dx_1 = \textrm{const}\int_{\mathbb{R}}|\widetilde{V}(x_1)|\,dx_1 \\&&\le \textrm{const}\; \underset{{n\in\mathbb{Z}}}\sup \,\mathcal{G}_n \\&& \leq \textrm{const}\parallel(\mathcal{G}_n)_{n\in\mathbb{Z}}\parallel_{1,w}.
\end{eqnarray*}
Hence \eqref{Est2} is equivalent to the following estimate
\begin{equation}\label{Est5}
N_-(q_{V, S})\leq 1 + C_{6}\left( \parallel(\mathcal{G}_n)_{n\in\mathbb{Z}}\parallel_{1,w} + \|V_{\ast}\|_{L_1\left(\mathbb{R}, L_{\mathcal{B}}(I)\right)}\right), \;\;\;\ \forall V \geq 0.
\end{equation}
Note the last term in right hand side of \eqref{Est5} (and \eqref{Est4}) drops out if $V$ does not depend on $x_2$.
}
\end{remark}
Let $\alpha > 0$ be given. It is well known that the lowest possible (semi-classical) rate of growth of $N_-(q_{\alpha V, S})$ is
$$
N_-(q_{\alpha V, S}) = O(\alpha) \mbox{ as } \alpha \longrightarrow  +\infty
$$
(see, e.g., \cite{BL,LM}).
It turns out that the finiteness of the first term in \eqref{Gest2} is necessary for $N_-(q_{\alpha V, S}) = O(\alpha) \mbox{ as } \alpha \longrightarrow  +\infty$ to hold (see next Theorem).
\begin{theorem}
Let $V \ge 0$. If $N_-(q_{\alpha V, S}) = O(\alpha) \mbox{ as } \alpha \longrightarrow  +\infty$, then $\|G_n\|_{1,w} < \infty$.
\begin{proof}
Consider the function
$$
w_n(x_1) := \left\{\begin{array}{cl}
  0 ,   & \  x_1 \le 2^{n-1}\,\textrm{ or }\, x_1\ge 2^{n+2} , \\ \\
  4(x_1 - 2 ^{n-1}) ,  & \ 2^{n-1} <  x_1 < 2^{n }  , \\ \\
 2^{n+1} , & \ 2^{n} \le x_1 \le  2^{n+1}, \\ \\
 2^{n+2} - x_1  ,  & \ 2^{n+1} <  x_1 < 2^{n+2}\,,
\end{array}\right.
$$
 for $n >0$. Let $u_n(x) = w_n(x_1)u_1(x_2)$, where $u_1$ is the first eigenfunction of the one-dimensional second order differential operator on $L^2((0, a))$ which is identically equal to $1$. Then  we have
\begin{eqnarray*}
\int_{S}|\nabla u_n(x)|^2\,dx &=& a\int_{\mathbb{R}}|w'_n(x_1)|^2dx_1\\&=&a\left( \int_{2^{n-1}}^{2^n}|w'_n(x_1)|^2dx_1 + \int_{2^{n}}^{2^{n+1}}|w'_n(x_1)|^2dx_1+\int_{2^{n+1}}^{2^{n+2}}|w'_n(x_1)|^2dx_1\right)\\ &=& C_72^{n+1}\,,
\end{eqnarray*}where
$$
C_{7} := 5a.
$$
Now
\begin{eqnarray*}
\int_S V(x)|u_n(x)|^2\,dx &\ge& \int_{2^n}^{2^{n+1}}2^{2(n+1)}\,dx_1\int_0^a V(x)\,dx_2\\ &\ge&\int_{2^n}^{2^{n+1}}2^{n+1}\,dx_1\int_0^a|x_1| V(x)\,dx_2\\ &=& \frac{1}{C_{7}}\left(\int_S|\nabla u_n(x)|^2 dx\right)\mathcal{G}_n\,.
\end{eqnarray*}
If $\mathcal{G}_n > C_{7}$, then $q_{V, S}[u_n] < 0$. The auxiliary functions $w_n$ can be defined similarly for $n \le 0$. Since $u_n$ and $u_k$ have disjoint supports for $|n - k| \ge 3$, then
$$
N_-\left(q_{V, S}\right) \ge \frac{1}{3}\textrm{card}\{n\in\mathbb{Z}\,:\, \mathcal{G}_n > C_{7}\}
$$ (see \cite[Theorem 9.1]{Eugene}). If for some constant $C_8>0$, $N_-\left(q_{\alpha V, S}\right) \le C_8\alpha$, then
$$
\frac{1}{3}\textrm{card}\{n\in\mathbb{Z}\,:\, \alpha \mathcal{G}_n > C_{7}\} \le C_8\alpha\,,
$$ and so
$$
\textrm{card}\left\{n\in\mathbb{Z}\,:\, \mathcal{G}_n > \frac{C_{7}}{\alpha}\right\} \le 3C_8\alpha\,.
$$ With $s = \frac{C_{7}}{\alpha}$ we have
$$
\textrm{card}\{n\in\mathbb{Z}\,:\, \mathcal{G}_n > s\} \le C_{9}s^{-1},\,\,\,s > 0,
$$ where $C_{9}:= 3 C_{7}C_8$.

\end{proof}
\end{theorem}

\section{Estimates involving norms of the potential supported by a Lipschitz curve inside a strip}\label{singular}
In this section we obtain estimates analogous to those in the previous section when the potential $V$ is strongly singular, i.e., when $V$ is supported by a Lipschitz curve $\ell$ embedded in $S$. When dealing with function spaces on $\ell$, we will always assume that $\ell$ is equipped with the arc length measure. Before we introduce the estimates, let us first look at the following operator that we shall need in the sequel:\\\\ Consider a one-dimensional Schr\"odinger operator $H_{X, \alpha}$, with point $\delta$-interactions on a countable set $X = \{x_k\}_{k =1}^{\infty}$ of points, called points of interaction and intensities $\alpha = \{\alpha_k\}_{k = 1}^{\infty}$ , defined by the differential expression $-\frac{d^2}{dx^2}$ on  functions $w(x)$ that belong to the Sobolev space $W^2_2(\mathbb{R}\setminus X)$ satisfying, in the points of the set $X$, the following conjugation conditions:
\begin{equation}\label{d1}
w(x_k + 0) = w(x_k - 0),\;\;\; w'(x_k + 0) - w'(x_k - 0) = \alpha_k w(x_k).
\end{equation}
Since for each $k$, $(x_k, x_{k+1})$ is an open interval, then any function in $W^2_2((x_k, x_{k+1}))$ and its derivative have well defined (one-sided) values at the end-points.
The operator $H_{X, \alpha}$ has the following representation
\begin{equation}\label{d2}
H_{X,\alpha} := -\frac{d^2}{dx^2} + \sum_{k = 1}^{\infty}\alpha_k \delta(x - x_k),
\end{equation}
where $\delta$ is the Dirac's delta distribution. We shall assume that $H_{X,\alpha}$ is self-adjoint (see, e.g., \cite{Alv}) in case the set $X$ is finite. One can also define the operator \eqref{d2} via its quadratic form  as follows
\begin{equation}\label{d3}
q[w] := \int_{\mathbb{R}}|w'(x)|^2\,dx - \sum_{k =1}^{\infty}\alpha_k|w(x_k)|^2, \;\;\;\forall w\in W^1_2(\mathbb{R}).
\end{equation}
\begin{lemma}\label{lemma1}
{\rm  Given an infinite sequence of positive numbers $(\alpha_k)$, there is a sequence of points $(x_k)$ in $X$ such that  \eqref{d2} has infinitely many negative eigenvalues.}
\end{lemma}
\begin{proof}
Let $\psi \in C_0^{\infty}(\mathbb{R})$ such that
\begin{eqnarray*}
& \psi(x) = \left\{\begin{array}{l}
  1 \ \mbox{ if } |x| < \frac{1}{2},   \\ \\
   0 \  \mbox{ if } |x| \geq 1  .
\end{array}\right.
\end{eqnarray*}
Assume that $x_{k + 1} - x_k > x_{k} - x_{k-1},$ for all $k = 1, 2, ...$ and let \\$\varphi_k(x) := \psi\left( \frac{2}{x_k - x_{k-1}}(x - x_k)\right)$. Then
\begin{eqnarray*}
& \varphi_k(x) = \left\{\begin{array}{l}
  1 \ \mbox{ if } |x - x_k| < \frac{x_k - x_{k-1}}{4} ,   \\ \\
   0 \  \mbox{ if } |x - x_k| \geq \frac{x_k - x_{k-1}}{2}  .
\end{array}\right.
\end{eqnarray*}
Let $\mathcal{L}$ be a linear subspace of $W^1_2(\mathbb{R})$ defined by
$$
\mathcal{L} := \left\{ w\in W^1_2(\mathbb{R})\; :\; w = \sum_{k = 1}^{\infty}a_k\varphi_k,\;a_k\in\mathbb{C}\right\}.
$$
Since $\varphi_k$ and $\varphi_j$ for $k \neq j$ have disjoint supports, then dim$\mathcal{L} = \infty$. Thus for all $w\in \mathcal{L}\setminus\{0\}$, it follows from \eqref{d3} that
\begin{eqnarray*}
q[w] &=& \sum_{k = 1}^{\infty}|a_k|^2\left(\int_{\mathbb{R}}|\varphi'_k(x)|^2\,dx - \alpha_k\right)\\
&=&\sum_{k = 1}^{\infty}|a_k|^2\left(\frac{4}{|x_k - x_{k-1}|^2}\int_{\mathbb{R}}|\psi'\left( \frac{2}{x_k - x_{k-1}}(x - x_k)\right)|^2\,dx - \alpha_k\right)\\
&=&\sum_{k = 1}^{\infty}|a_k|^2\left(\frac{2}{|x_k - x_{k-1}|}\int_{\mathbb{R}}|\psi'(t)|^2\,dt - \alpha_k\right),
\end{eqnarray*}where $ t = \frac{2}{x_k - x_{k-1}}(x - x_k)$. Take $x_k$ such that
$$
x_k - x_{k-1} > \frac{2}{\alpha_k}\int_{\mathbb{R}}|\psi'(t)|^2\,dt,
$$then $q[w] < 0$ and the operator \eqref{d2} has infinitely many negative eigenvalues.

\end{proof}
Let us now return to the operator \eqref{self-adj} with $V$ supported by and locally integrable on a Lipschitz curve $\ell$ embedded in $S$. Let
\begin{eqnarray*}
q_{V,\ell}[u] &:=& \int_S |\nabla u(x)|^2dx - \int_{\ell}V(x)|u(x)|^2\, ds(x)\\
\textrm{Dom}(q_{V,\ell})& =& W^1_2(S) \cap L^2(\ell, Vds).
\end{eqnarray*}
Let $\{x_1^{(k)}\}, \,k\in\mathbb{N}$ be a sequence of points on $\mathbb{R}$ satisfying conditions in \eqref{d1}. Define
$$
\gamma_k := \ell \cap \left(\{x^{(k)}_1\}\times (0, a)\right),\;\;x_1\in\mathbb{R}, \;\;k\in\mathbb{N}
$$
and
$$
\Sigma := \left\{x^{(k)}_1 \;:\; |\gamma_k| > 0\right\}.
$$
Then the set $\Sigma$ is at most countable. Let $I$ be an arbitrary interval in $\mathbb{R}$ and let
$$
\nu(I) := \int_{\ell\cap \left(I\times(0, a)\right)}V(x)\,ds(x).
$$
Furthermore, let
\begin{eqnarray*}
\mathcal{F}_n &:=& \int_{I_n} |x_1| \,d\nu(x_1) , \ n \not= 0 , \ \ \ F_0 := \int_{I_0} d\nu(x_1)\,\,\,\, (cf. \eqref{calAn}),\\
\ell_n &:=& \ell \cap S_n \,,\,\,S_n := (n, n+1)\times (0, a),\,\,\,n\in\mathbb{N},\\
\mathcal{C}_n &:=& \|V\|^{(\textrm{av})}_{\mathcal{B}, \ell_n}.
\end{eqnarray*}
\begin{theorem}
Suppose that $N$ is the cardinality of  $\Sigma$. Then there exist constants $c_1, c_2, C_{10}, C_{11} > 0$ such that
\begin{equation}\label{Gest3}
N_-(q_{V, \ell})\leq 1 + N + C_9\sum_{\{n\in\mathbb{Z}, \mathcal{F}_n > c_1\}}\sqrt{\mathcal{F}_n} + C_{10}\sum_{\{n\in\mathbb{Z}, \mathcal{C}_n > c_2\}}\mathcal{C}_n, \;\;\;\;\forall V\geq 0.
\end{equation}
If $\Sigma$ is infinite, then $N_-(q_{V, \ell}) = \infty$.
\end{theorem}
\begin{proof}
Let $q_{1,2V, \ell}$ and $q_{2, 2V,\ell}$ be the restrictions of the form $q_{V, \ell}$ to the subspaces $\mathcal{H}_1$ and $\mathcal{H}_2$ respectively (see \eqref{H}), then
\begin{equation}\label{es1}
N_-(q_{V,\ell}) \le N_-(q_{1, 2V,\ell})) + N_-(q_{2, 2V,\ell})
\end{equation} (cf. \eqref{estimate}).
Let us start by estimating the first term in the right-hand side of \eqref{es1}.
On the complement of $\Sigma$, $\nu(\{x_1\}) = 0$ for all $x_1\in\mathbb{R}$. This implies
\begin{eqnarray*}
\int_{\ell}V(x)|u(x)|^2\,ds(x) &=& \int_{\mathbb{R}} |w(x_1)|^2\,d\nu(x_1) + |w(x_1)|^2\underset{k\in\mathbb{N}}\sum\int_{\gamma_k}V(x_2)\,dx_2 \\&=& \int_{\mathbb{R}}|w(x_1)|^2\,d\nu(x_1) + \underset{k\in\mathbb{N}}\sum c_k|w(x_1)|^2,
\end{eqnarray*}
where $ c_k := \int_{\gamma_k}V(x_2)\,dx_2 < \infty$.
Hence
\begin{eqnarray*}
q_{1,2V,\ell}[u] &=& \int_{S}|\nabla u(x)|^2 \,dx  - 2\int_{\ell}V(x)|u(x)|^2\,dx\\ &=& \int_{\mathbb{R}}|w'(x_1)|^2\,dx_1 - 2\int_{\mathbb{R}}|w(x_1)|^2\,d\nu(x_1) - \underset{k\in\mathbb{N}}\sum c'_k|w(x_1)|^2,
\end{eqnarray*}where $c'_k := 2c_k$. Let
\begin{eqnarray*}
q_{1,2\nu}[w] &:=& \int_{\mathbb{R}}|w'(x_1)|^2\,dx_1 - 2\int_{\mathbb{R}}|w(x_1)|^2\,d\nu(x_1),\\
\textrm{Dom}(q_{1, 2\nu}) &=& W^1_2(\mathbb{R})\cap L^2\left(\mathbb{R}, d\nu\right),\\
q_{2,c'_k}[w] &:=& \int_{\mathbb{R}}|w'(x_1)|^2\,dx_1 -  \underset{k\in\mathbb{N}}\sum c'_k|w(x_1)|^2,\\
\textrm{Dom}(q_{2, c'_k}) &=& W^1_2(\mathbb{R}).
\end{eqnarray*}
Then, it follows from \cite[Lemma 3.6]{Grig} that
\begin{equation}\label{es2}
N_-(q_{1, 2V,\ell}) \le N_-(q_{1, 2\nu}) + N_-(q_{2, c'_k}).
\end{equation}
Similarly to \eqref{Est1*} one has
\begin{equation}\label{es3}
N_-(q_{1, 2\nu}) \le 1 + 7.16\sum_{\{n\in\mathbb{Z}, \mathcal{F}_n > 0.046\}}\sqrt{\mathcal{F}_n}\,.
\end{equation}
If $\Sigma$ is finite, then
\begin{equation}\label{es4}
N_-(q_{2, c'_k}) \le N
\end{equation}
(see, e.g., \cite{Alv}), otherwise, Lemma \ref{lemma1} implies
\begin{equation}\label{es5}
N_-(q_{2, c'_k}) = \infty .
\end{equation}
Now, it remains to estimate the second term in the right-hand side of \eqref{es1}.
Let
\begin{eqnarray*}
q_{2, 2V,\ell_n}[u] &:=& \int_{S_n} |\nabla u(x)|^2dx - \int_{\ell_n}V(x)|u(x)|^2\, ds(x)\\
\textrm{Dom}(q_{2,2V,\ell_n})& =& \{u\in W^1_2(S_n) \cap L^2(\ell_n, Vds)\, :\, \int_{S_n}u(x)\,dx =0\}.
\end{eqnarray*}
For any $V\in L_{\mathcal{B}}(\ell_n), \, V\ge 0$ and any $r\in\mathbb{N}$, following a similar argument in the proof of \cite[Lemma 7.6]{Eugene}, one can show that there exists a finite cover of $\ell_n$ by rectangles $S_{n_k},\,k = 1, ..., r_0$ such that $r_0 \le r$ and
\begin{equation}\label{eq1}
\int_{\ell_n}V(x)|u(x)|^2\,ds(x) \le C_{12}r^{-1}\|V\|^{(\textrm{av})}_{\mathcal{B}, \ell_n}\int_{S_n}|\nabla u(x)|^2\,dx
\end{equation}
for all $u\in W^1_2(S_n)\cap L^2(\ell_n, Vds)$ with $\int_{S_{n_k}}u(x)\,dx =0$, where $C_{12}$ is a constant independent of $V$.
Now, using \eqref{eq1} instead of \cite[Lemma 7.6]{Eugene} in the proof of \cite[Lemma 7.7]{Eugene}, one can easily show similarly to \cite[Lemma 7.8]{Eugene} that there is a constant $C_{11} > 0$ such that
$$
N_-(q_{2, 2V, \ell_n}) \le C_{11}\|V\|^{(\textrm{av})}_{\mathcal{B}, \ell_n}\,,\,\,\,\forall V \ge 0.
$$
If $\|V\|^{(\textrm{av})}_{\mathcal{B}, \ell_n} < \frac{1}{C_{11}}$, then $N_-(q_{2, 2V, \ell_n}) = 0$. Thus, similarly to \eqref{estimate2} one has that for any $c_2 < \frac{1}{C_{11}}$ that
\begin{equation}\label{es6}
N_-(q_{2, 2V, \ell}) \le C_{11}\underset{\{\mathcal{C}_n > c_2, \,n\in\mathbb{Z}\}}\sum \mathcal{C}_n\,,\,\,\,\forall V\ge 0.
\end{equation}
Hence, the statement of the Theorem follows from \eqref{es1}, \eqref{es2}, \eqref{es3}, \eqref{es4}, \eqref{es5} and \eqref{es6}.
\end{proof}

\section*{Competing interests}
The author declares that he has no competing interests.
\section*{Acknowledgment}
This research was supported by the Commonwealth Scholarship Commission in the United Kingdom, grant UGCA-2013-138. The author is also grateful to the reviewers for their helpful comments and suggestions.

\end{document}